\documentclass[a4paper,twoside,11pt]{amsart}  
\usepackage{amsmath, amsthm, amssymb} 

\usepackage{amsrefs}
\usepackage[inner=2.4cm,outer=4cm,top=2.4cm,bottom=2.4cm,marginparwidth=90pt]{geometry} %the geometry package allows different page sizes, margins etc. 

\usepackage{times} 
\usepackage{enumitem}
 
\DeclareMathOperator{\trace}{tr} 
\newcommand{\RR}{\mathbb{R}}
\newcommand{\R}{{\mathbb R}}
\DeclareMathOperator{\Ric}{Ric} 
\newcommand{\sk}{\mathbf{S_\kappa}} 
\newcommand{\ck}{\mathbf{C_\kappa}} 
\newcommand{\tk}{\mathbf{T}_\kappa}
 \newcommand{\E}{\mathrm e}
\newtheorem{theorem}{Theorem}
 
\newtheorem{proposition}[theorem]{Proposition}

\begin{document}

\title[eigenvalues and modulus of continuity for parabolic equations]{Sharp modulus of continuity for parabolic equations on manifolds and lower bounds for the first eigenvalue} 

\author{Ben Andrews} 
\address{Mathematical Sciences Institute, Australia National University; Mathematical Sciences Center, Tsinghua University; and Morningside Center for Mathematics, Chinese Academy of Sciences.}
\email{Ben.Andrews@anu.edu.au}
\thanks{Supported by Discovery Projects grants DP0985802 and DP120102462 of the Australian Research Council.}
\author{Julie Clutterbuck} 
\address{Mathematical Sciences Institute, Australia National University}
\email{Julie.Clutterbuck@anu.edu.au}
\address{}

\date{\today}  
 
\begin{abstract}   We derive sharp estimates on modulus of continuity for solutions of the heat equation on a compact Riemannian manifold with a Ricci curvature bound, in terms of initial oscillation and elapsed time.   As an application, we give an easy proof of the optimal lower bound on the first eigenvalue of the Laplacian on such a manifold as a function of diameter.
\end{abstract}
 
\maketitle 
 
\section{Introductory comments}

In our previous papers \cites{AC1,AC2} we proved sharp bounds on the modulus of continuity of solutions of various parabolic boundary value problems on domains in Euclidean space.  In this paper, our aim is to extend these estimates to parabolic equations on manifolds.  Precisely, let $(M,g)$ be a compact Riemannian manifold with induced distance function $d$, diameter $\sup\{d(x,y):\ x,y\in M\}=D$ and lower Ricci curvature bound $\text{\rm Ric}(v,v)\geq (n-1)\kappa g(v,v)$.  Let $a:\ T^*M\to\text{\rm Sym}_2\left(T^*M\right)$ be a parallel equivariant map (so that $a(S^*\omega)(S^*\mu,S^*\nu)=a(\omega)(\mu,\nu)$ for any $\omega$, $\mu$, $\nu$ in $T_x^*M$ and $S\in O(T_xM)$, while $\nabla\left(a(\omega)(\mu,\nu)\right)=0$ whenever $\nabla\omega=\nabla\mu=\nabla\nu=0$).   Then we consider solutions to the parabolic equation
\begin{align}\label{eq:flow}
\dfrac{\partial u}{\partial t} &= a^{ij}(Du)\nabla_i\nabla_ju. 
\end{align} 
Our assumptions imply that the coefficients $a^{ij}$ have the form
\begin{equation}\label{eq:formofa}
a(Du)(\xi,\xi) =\alpha(|Du|)\frac{\left(Du\cdot\xi\right)^2}{|Du|^2} + \beta(|Du|)\left(|\xi|^2-\frac{\left(Du\cdot\xi\right)^2}{|Du|^2}\right)
\end{equation}
for some smooth positive functions $\alpha$ and $\beta$.  Of particular interest are the cases of the heat equation (with $\alpha=\beta=1$) and the $p$-laplacian heat flows (with $\alpha=(p-1)|Du|^{p-2}$ and $\beta = |Du|^{p-2}$).
Here we are principally concerned with the case of manifolds without boundary, but can also allow $M$ to have a nontrivial convex boundary (in which case we impose Neumann boundary conditions $D_\nu u=0$).  Our main aim is to provide the following estimates on the modulus of continuity of solutions in terms of the initial oscillation, elapsed time, $\kappa$ and $D$:

\begin{theorem}[Modulus of continuity estimate]\label{thm:moc}
Let  $(M,g)$ be a compact Riemannian manifold (possibly with smooth, uniformly locally convex boundary) with diameter $D$ and Ricci curvature bound $\Ric\ge (n-1)\kappa g$ for some constant $\kappa\in\RR$.   
 Let $u: M\times [0,T)\rightarrow \R$ be a smooth solution to equation \eqref{eq:flow}, with Neumann boundary conditions if $\partial M\neq\emptyset$.
Suppose that  
\begin{itemize}
\item $u(\cdot,0)$ has a smooth modulus of continuity $\varphi_0:[0,D/2]\rightarrow\R$ with $\varphi_0(0)=0$ and $\varphi_0'\geq 0$;
\item $\varphi:[0,D/2]\times \RR_+\rightarrow \RR$ satisfies
\begin{enumerate}[label=\rm{(\roman*)}]
\item $\varphi(z,0)=\varphi_0(z)$ for each $z\in[0,D/2]$;
\item\label{1deqn} $\frac{\partial \varphi}{\partial t}\ge \alpha(\varphi')\varphi'' - (n-1) \tk \beta(\varphi')\varphi'$;
\item $\varphi'\geq 0$ on $[0,D/2]\times \RR_+$.
\end{enumerate}
\end{itemize}

Then $\varphi(\cdot,t)$ is a modulus of continuity for $u(\cdot,t)$ for each $t\in[0,T)$:
$$|u(x,t)-u(y,t)|\le 2\varphi\left( \frac{d(x,y)}2,t\right).$$  
\end{theorem}

Here we use the notation
\begin{equation}  \label{defn of ck}
\ck(\tau )=\begin{cases} \cos\sqrt{\kappa}\tau , & \kappa>0 \\ 
1, &\kappa =0 \\ 
\cosh \sqrt{-\kappa} \tau , & \kappa <0, 
\end{cases} % \end{equation*} 
\quad \text{ and } \quad 
%\begin{equation*} 
\sk(\tau )=\begin{cases} \frac1{\sqrt{\kappa}}\sin\sqrt{\kappa}\tau , & \kappa>0 \\ 
\tau , &\kappa =0 \\ 
\frac1{\sqrt{-\kappa}}\sinh \sqrt{-\kappa} \tau , & \kappa <0,
\end{cases} \end{equation} 
and  $$
\tk(s) := \kappa\frac{\sk(s)}{\ck(s)}=\begin{cases}
\sqrt{\kappa}\tan\left(\sqrt{\kappa}s\right),&\kappa>0\\
0,& \kappa=0\\
-\sqrt{-\kappa}\tanh\left(\sqrt{-\kappa}s\right),&\kappa<0.
\end{cases}
$$

%We remark that in most circumstances the strict inequality $\varphi>0$ can be replaced by a weak one, since a function $\varphi$ satisfying the required differential inequality with $\varphi\geq 0$ can be approximately by examples with strictly positive $\varphi$.   This is the case for the heat equation, or for any example where $\alpha$ and $\beta$ do not degenerate as their argument approaches zero.  Some care may be required, however, for examples such as the $p$-Laplacian heat flows where $\alpha$ and $\beta$ may degenerate or become singular as the gradient approaches zero.
%

 These estimates are sharp, holding exactly for certain symmetric solutions on particular warped product spaces.  The modulus of continuity estimates also imply sharp gradient bounds which hold in the same situation.  The central ingredient in our argument is a comparison result for the second derivatives of the distance function (Theorem \ref{thm:dist-comp}) which is a close relative of the well-known Laplacian comparison theorem.  We remark that the assumption of smoothness can be weakened:  For example in the case of the $p$-laplacian heat flow we do not expect solutions to be smooth near spatial critical points, but nevertheless solutions are smooth at other points and this is sufficient for our argument.

As an immediate application of the modulus of continuity estimates, we provide a new proof of the optimal lower bound on the smallest positive eigenvalue of the Laplacian in terms of $D$ and $\kappa$:
Precisely, if we define 
$$
\lambda_1(M,g) = \inf\left\{\int_M |Du|_g^2\,d\text{\rm Vol}(g):\ \int_M u^2d\text{\rm Vol}(g)=1,\ \int_Mu\,d\text{\rm Vol}(g)=0\right\},
$$
and
$$
\lambda_1(D,\kappa,n) = 
\inf\left\{\lambda_1(M,g):\ \text{\rm dim}(M)=n,\ \text{\rm diam}(M)\leq D,\ \text{\rm Ric}\geq (n-1)\kappa g\right\},
$$
then we characterise $\lambda_1(D,\kappa)$ precisely as the first eigenvalue of a certain one-dimensional Sturm-Liouville problem:

\begin{theorem}[Lower bound on the first eigenvalue] \label{first eigenvalue estimate} 
Let $\mu$ be the first eigenvalue of the Sturm--Liouville problem 
\begin{gather} \label{SL equation}  \begin{split}
\frac1{\ck^{n-1}}\left(\Phi' \ck^{n-1}\right)' +\mu\Phi&=0 \text{ on }[-D/2,D/2],\\
\Phi'(\pm D/2 )&=0.   
\end{split}
\end{gather}
Then $\lambda_1(D,\kappa,n)=\mu$.
\end{theorem}

   Previous results in this direction include the results derived from gradient estimates due to Li \cite{Li-ev} and Li and Yau \cite{LiYau}, with the sharp result for non-negative Ricci curvature first proved by Zhong and Yang \cite{ZY}.  The complete result as stated above is implicit in the results of Kr\"oger \cite{Kroeger}*{Theorem 2} and explicit in those of Bakry and Qian \cite{BakryQian}*{Theorem 14}, which are also based on gradient estimate methods.  Our contribution is the rather simple proof using the long-time behaviour of the heat equation (a method which was also central in our work on the fundamental gap conjecture \cite{AC3}, and which has also been employed successfully in \cite{Ni}) which seems considerably easier than the previously available arguments.  In particular the complications arising in previous works from possible asymmetry of the first eigenfunction are avoided in our argument.    A similar argument proving the sharp lower bound for $\lambda_1$ on a Bakry-Emery manifold may be found in \cite{Andrews-Ni}. 

The estimate in Theorem \ref{first eigenvalue estimate} is sharp (that is, we obtain an equality and not just an inequality), since for a given diameter $D$ and Ricci curvature bound $\kappa$, we can construct a sequence of manifolds satisfying these bounds on which the first eigenvalue approaches $\mu_1$ (see the remarks after Corollary 1 in \cite{Kroeger}).  We include a discussion of these examples in section \ref{sec:examples}, since the examples required for our purposes are a simpler subset of those constructed in \cite{Kroeger}.   We also include in section \ref{sec:Li} a discussion of the implications for a conjectured inequality of Li.

\section{A comparison theorem for the second derivatives of distance}\label{sec:dist-comp}

\begin{theorem}   \label{thm:dist-comp}
Let $(M,g)$ be a complete connected Riemannian manifold with a lower Ricci curvature bound ${\mathrm{Ric}}\geq (n-1)\kappa g$, and let $\varphi$ be a smooth function with $\varphi'\geq0$.   Then on $(M\times M)\setminus \{(x,x):\ x\in M\}$ the function $v(x,y)=2\varphi(d(x,y)/2)$ is a viscosity supersolution of 
$$
\mathcal{L}[\nabla^2 v,\nabla v]=2\left[\alpha(\varphi')\varphi''-(n-1)\tk\beta(\varphi')\varphi'\right]\big|_{d/2},
$$
where 
$$
\mathcal{L}[B,\omega] = \inf\left\{\trace(AB):\ 
\begin{aligned}
&A\in\text{\rm Sym}_2(T^*_{x,y}(M\times M))\\
&A\geq0\\
&A|_{T^*_xM}=a(\omega\big|_{T_xM})\\
&A|_{T^*_yM}=a(\omega\big|_{T_yM})
\end{aligned}\right\}
$$
for any $B\in\text{\rm Sym}_2(T_{x,y}(M\times M)$ and $\omega\in T^*_{(x,y)}(M\times M)$.
\end{theorem}

\begin{proof}
By approximation it suffices to consider the case where $\varphi'$ is strictly positive.
Let $x$ and $y$ be fixed, with $y\neq x$ and $d=d(x,y)$, and let $\gamma:\ [-d/2,d/2]\to M$ be a minimizing geodesic from $x$ to $y$ (that is, with $\gamma(-d/2)=x$ and $\gamma(d/2)=y$) parametrized by arc length.  Choose an orthonormal basis $\{E_i\}_{1\leq i\leq n}$ for $T_xM$ with $E_n=\gamma'(-d/2)$.  Parallel transport along $\gamma$ to produce 
an orthonormal basis $\{E_i(s)\}_{1\leq i\leq n}$ for $T_{\gamma(s)}M$ with $E_n(s)=\gamma'(s)$ for each $s\in [-d/2,d/2]$.  Let $\{E_*^i\}_{1\leq i\leq n}$ be the dual basis for $T^*_{\gamma(s)}M$.  

To prove the theorem, consider any smooth function  $\psi$ defined on a neighbourhood of $(x,y)$ in $M\times M$ such that $\psi\leq v$ and $\psi(x,y)=v(x,y)$.    %%% 
We must prove that $\mathcal{L}[\nabla^2\psi,\nabla\psi]\big|_{(x,y)}\leq 2\left[\alpha(\varphi')\varphi''-(n-1)\beta(\varphi')\varphi'\tk\right]\big|_{d(x,y)/2}$. 
By definition of $\mathcal{L}$ it suffices to find a non-negative $A\in\text{\rm Sym}_2(T^*_{x,y}(M\times M))$ such that $A|_{T_xM}=a\left(\nabla\psi|_{T_xM}\right)$ and $A|_{T_yM}=a\left(\nabla\psi|_{T_yM}\right)$, with $\trace(AD^2\psi)\leq 2\left[\alpha(\varphi')\varphi''-(n-1)\beta(\varphi')\varphi'\tk\right]\big|_{d/2}$.   

Before choosing this we observe that $\nabla\psi$ is determined by $d$ and $\varphi$:  We have $\psi\leq 2\varphi\circ d/2$ with equality at $(x,y)$.  In particular we have (since $\varphi$ is nondecreasing)
$$
\psi(\gamma(s),\gamma(t))\leq 2\varphi(d(\gamma(s),\gamma(t))/2)\leq 2\varphi(L[\gamma|_{[s,t]}]/2)\leq 2\varphi(|t-s|/2),
$$
for all $s\neq t$, with equality when $t=d/2$ and $s=-d/2$.  This gives $\nabla\psi(E_n,0)=-\varphi'(d/2)$ and $\nabla\psi(0,E_n)=\varphi'(d/2)$.  To identify the remaining components of $\nabla\psi$, we define $\gamma^y_i(r,s) = \exp_{\gamma(s)}(r(1/2+s/d)E_i(s))$ for $1\leq i\leq n-1$.  Then we have
$$
\psi(x,\exp_y(rE_i))\leq 2\varphi(L[\gamma_i^y(r,.)]/2)
$$
with equality at $r=0$.  The right-hand side is a smooth function of $r$ with derivative zero, from which it follows that $\nabla\psi(0,E_i)=0$.   Similarly we have $\nabla\psi(E_i,0)=0$ for $i=1,\dots,n-1$.  Therefore we have
$$
\nabla\psi\big|_{(x,y)} = \varphi'(d(x,y)/2) (-E^n_*,E^n_*).
$$
In particular we have by \eqref{eq:formofa}
$$
a(\nabla\psi|_{T_xM}) = \alpha(\varphi')E_n\otimes E_n+\beta(\varphi')\sum_{i=1}^{n-1}E_i\otimes E_i,
$$
and similarly for $y$.

Now we choose $A$ as follows:
\begin{equation}   \label{choice of A}
A = \alpha(\varphi')(E_n,-E_n)\otimes(E_n,-E_n)+\beta(\varphi')\sum_{i=1}^{n-1}(E_i,E_i)\otimes(E_i,E_i).
\end{equation}  
This is manifestly non-negative, and agrees with $a$ on $T_xM$ and $T_yM$ as required.
This choice gives
\begin{equation}\label{eq:trace}
\trace(A\nabla^2\psi) = \alpha(\varphi)\nabla^2\psi\left((E_n,-E_n),(E_n,-E_n)\right)+\beta(\varphi')\sum_{i=1}^{n-1}\nabla^2\psi\left((E_i,E_i),(E_i,E_i)\right).
\end{equation}
For each $i\in\{1,\dots,n-1\}$ let
$\gamma_{i}:\ (-\varepsilon,\varepsilon)\times[-d/2,d/2]\to M$ be any smooth one-parameter family of curves with $\gamma_{i}(r,\pm d/2) = \exp_{\gamma(\pm d/2)}(rE_i(\pm d/2))$ for $i=1,\dots,n-1$, and $\gamma_i(0,s)=\gamma(s)$.  Then $d(\exp_x(rE_i),\exp_y(rE_i))\leq L[\gamma_i(r,.)]$ and hence
\begin{align*}
\psi(\exp_x(rE_i),\exp_y(rE_i))&\le v(\exp_x(rE_i),\exp_y(rE_i))\\
&=2\varphi\left(\frac{d(\exp_x(rE_i),\exp_y(rE_i))}{2}\right)\\
&\leq 2\varphi\left(\frac{L[\gamma_i(r,.)]}{2}\right)  %%%
\end{align*}
since $\varphi$ is nondecreasing.  Since the functions on the left and the right are both smooth functions of $r$ and equality holds for $r=0$, it follows that
\begin{equation}\label{eq:D2i}
\nabla^2\psi((E_i,E_i),(E_i,E_i))\leq2\sum_{i=1}^{n}\frac{d^2}{dr^2}\left(\varphi\left(\frac{L[\gamma_{i}(r,.)]}{2}\right)\right)
\Big|_{r=0}.
\end{equation}
Similarly, since $d-2r=L[\gamma\big|_{[-d/2+r,d/2-r]}]\geq d(\gamma(-d/2+r),\gamma(d/2-r))$ we have  \begin{equation}\label{eq:D2n}
\nabla^2\psi(E_n,-E_n),(E_n,-E_n))\leq 2\frac{d^2}{dr^2}\left(\varphi\left(\frac{d}{2}-r\right)\right)\Big|_{r=0} = 2\varphi''\left(\frac{d}{2}\right).
\end{equation}
Now we make a careful choice of the curves $\gamma_i(r,.)$ motivated by the situation in the model space, in order to get a useful result on the right-hand side in the inequality \eqref{eq:D2i}:  
To begin with if $K>0$ then we assume that $d<\frac{\pi}{\sqrt{K}}$ (we will return to deal with this case later).
We choose
$$
\gamma_i(r,s) = \exp_{\gamma(s)}\left(\frac{r\ck(s)E_i}{\ck(d/2)}\right),
$$
where  $\ck$ is given by \eqref{defn of ck}.
%$$
%\ck(s) = 
%\begin{cases}
%\cos\left(\sqrt{K}s\right),& K>0;\\
%1,&K=0;\\
%\cosh\left(\sqrt{-K}s\right),&K<0.
%\end{cases}
%$$
Now we proceed to compute the right-hand side of \eqref{eq:D2i}:  Denoting $s$ derivatives of $\gamma_i$ by $\gamma'$ and $r$ derivatives by $\dot\gamma$, we find
\begin{align*}
\frac{d}{dr}\left(\frac{L[\gamma_{i}(r,.)]}{2}\right)
&=\frac{d}{dr}\left(
\int_{-d/2}^{d/2}\left\|\gamma'(r,s)\right\|\,ds\right)\\
&=\int_{-d/2}^{d/2}\frac{\left\langle \gamma',\nabla_r\gamma'\right\rangle}{\|\gamma'\|}\,ds.
\end{align*}
In particular this gives zero when $r=0$.
Differentiating again we obtain 
(using $\|\gamma'(0,s)\|=1$ and the expression $\dot\gamma(0,s)=\frac{\ck(s)}{\ck(d/2)}E_i$)
$$
\frac{d^2}{dr^2}\left(\frac{L[\gamma_{i}(r,.)]}{2}\right)\Big|_{r=0}
=\int_{-d/2}^{d/2}\|\nabla_r\gamma'\|^2-\left\langle\gamma',\nabla_r\gamma'\right\rangle^2+
\left\langle \gamma',\nabla_r\nabla_r\gamma'\right\rangle\,ds.
$$
Now we observe that $\nabla_r\gamma'=\nabla_s\dot\gamma = \nabla_s\left(\frac{\ck(s)}{\ck(d/2)}E_i\right) = \frac{\ck'(s)}{\ck(d/2)}E_i$, while 
$$
\nabla_r\nabla_r\gamma' = \nabla_r\nabla_s\dot\gamma 
=\nabla_s\nabla_r\dot\gamma -R(\dot\gamma,\gamma')\dot\gamma
=-\frac{\ck(s)^2}{\ck(d/2)^2} R(E_i,E_n)E_i,
$$
since by the definition of $\gamma_i(r,s)$ we have $\nabla_r\dot\gamma=0$.
This gives
$$
\frac{d^2}{dr^2}\left(\frac{L[\gamma_{i}(r,.)]}{2}\right)\Big|_{r=0}
=\frac{1}{\ck(d/2)^2}\int_{-d/2}^{d/2}\left\{\ck'(s)^2-\ck(s)^2 R(E_i,E_n,E_i,E_n)\right\}\,ds.
$$
Summing over $i$ from $1$ to $n-1$ gives
\begin{align*}
\left.\sum_{i=1}^{n-1}\frac{d^2}{dr^2}\left(\frac{L[\gamma_{i}(r,.)]}{2}\right)\right|_{r=0}
&=\frac{1}{\ck(d/2)^2}\int_{-d/2}^{d/2}\left\{(n-1)\ck'(s)^2 - \ck(s)^2\sum_{i=1}^{n-1}R(E_i,E_n,E_i,E_n)\right\}\,ds\\
&=\frac{1}{\ck(d/2)^2}\int_{-d/2}^{d/2}\left\{(n-1)\ck'(s)^2 - \ck(s)^2{\mathrm{Ric}}(E_n,E_n)\right\}\,ds\\
&\leq \frac{n-1}{\ck(d/2)^2}\int_{-d/2}^{d/2}\left\{\ck'(s)^2 - \kappa \ck(s)^2\right\}\,ds.
\end{align*}

In the case $\kappa=0$ the integral is zero; in the case $\kappa<0$, or the case $\kappa>0$ with $d<\frac{\pi}{\sqrt{\kappa}}$, we have
% original version
%\begin{align*}
%\frac{1}{\ck(d/2)^2}\int_{-d/2}^{d/2}\left\{\ck'(s)^2 - K\ck(s)^2\right\}\,ds
%&= \frac{1}{\cos^2\left(\frac{\sqrt{K}d}{2}\right)}\int_{-d/2}^{d/2}\left(K\sin^2\left(\sqrt{K}s\right)-K\cos^2
%\left(\sqrt{K}s\right)\right)\,ds\\
%&=-\frac{K}{\cos^2\left(\frac{\sqrt{K}d}{2}\right)}\int_{-d/2}^{d/2} \cos\left(2\sqrt{K}s\right)\,ds\\
%&=-\frac{\sqrt{K}}{2\cos^2\left(\frac{\sqrt{K}d}{2}\right)}\sin\left(2\sqrt{K}s\right)\Big|_{-d/2}^{d/2}\\
%&=-2\sqrt{K}\tan\left(\frac{\sqrt{K}d}{2}\right)\\
%&=-2f_\kappa\left(\frac{d}{2}\right).
%\end{align*}
\begin{align*}
\frac{1}{\ck(d/2)^2}\int_{-d/2}^{d/2}\left\{\ck'(s)^2 - \kappa\ck(s)^2\right\}\,ds
&= \frac{1}{\ck(d/2)^2}\int_{-d/2}^{d/2} \left(-\kappa \sk \ck'- \kappa \sk' \ck    \right) \,ds\\
&=-\frac{\kappa}{\ck(d/2)^2}    \int_{-d/2}^{d/2} \left(\ck\sk  \right)' \,ds\\
%&=-\frac{\kappa}{\ck(d/2)^2}   \left[\ck\sk  \right]_{-d/2}^{d/2} \\
&=-\frac{2\kappa \ck(d/2)\sk(d/2)      }{\ck(d/2)^2}  \\
%&= -2\kappa \frac{\sk(d/2)      }{\ck(d/2)}  \\
&=-2\tk(d/2).
\end{align*}
Finally, we have
$$
\left.\frac{d}{dr}\left(\varphi\left(\frac{L[\gamma_{i}(r,.)]}{2}\right)\right)
\right|_{r=0} =\left. \varphi'\frac{d}{dr}\left(\frac{L[\gamma_{i}(r,.)]}{2}\right)
\right|_{r=0} = 0,
$$
and so
\begin{align*}
\sum_{i=1}^{n-1}\frac{d^2}{dr^2}\left(\varphi\left.\left(\frac{L[\gamma_{i}(r,.)]}{2}\right)\right)
\right|_{r=0} &= \sum_{i=1}^{n-1}\left.\left(\varphi'\frac{d^2}{dr^2}\left(\frac{L[\gamma_{i}(r,.)]}{2}\right)
\right|_{r=0} +\varphi''\left(\frac{d}{dr}\left.\left(\frac{L[\gamma_{i}(r,.)]}{2}\right)
\right|_{r=0}\right)^2\right)\\
&\leq -2(n-1)\left.\varphi'\tk\right|_{d/2}.
\end{align*}
Now using the inequalities \eqref{eq:D2i} and \eqref{eq:D2n}, we have from \eqref{eq:trace} that
\begin{equation}   \label{trace inequality}
\mathcal{L}[\nabla^2\psi,\nabla\psi]\leq{\mathrm{trace}}\left(A\nabla^2\psi\right)\leq 2\left[\alpha(\varphi')\varphi''-(n-1)\beta(\varphi')\varphi'\tk\right]\big|_{d/2},
\end{equation}
as required.

In the case $d=\frac{\pi}{\sqrt{K}}$ then we choose instead $\gamma_i(r,s) = \exp_{\gamma(s)}
\left(\frac{r \mathbf{C_{\kappa'}} (s)E_i}{\mathbf{C_{\kappa'}}(d/2)}\right)$, for arbitrary $\kappa'<\kappa$.  Then the computation above gives
$$
\sum_{i=1}^{n-1}\nabla^2\psi((E_i,E_i),(E_i,E_i))\leq -2(n-1)\varphi'\tk.
$$
Since the right hand side approaches $-\infty$ as $\kappa'$ increases to $\kappa$, we have a contradiction to the assumption that $\psi$ is smooth.  Hence no such $\psi$ exists and there is nothing to prove.
\end{proof}

\section{Estimate on the modulus of continuity for solutions of heat equations} \label{mfld without bdy section}
In this section we prove Theorem \ref{thm:moc}, which extends the oscillation estimate from domains in $\mathbb{R}^n$ to compact Riemannian manifolds.   The estimate is analogous to \cite{AC2}*{Theorem 4.1}, the modulus of continuity estimate for the Neumann problem on a convex Euclidean domain.
 
\begin{proof}[Proof of Theorem \ref{thm:moc}]
Recall that $(M,g)$ is a compact Riemannian manifold, possibly with boundary (in which case we assume that the boundary is locally convex).
Define an evolving quantity, $Z$, on the product manifold  ${M}\times{M}\times[0,\infty)$:
\[  
Z(x,y,t):=u(y,t)-u(x,t)-2\varphi(d(x,y)/2,t)-\epsilon(1+t)
\] 
for small $\epsilon>0$.  

We have assumed that  $\varphi$ is a modulus of continuity for $u$ at $t=0$, and so $Z(\cdot,\cdot,0)\leq -\epsilon<0$.  Note also that $Z$ is smooth on $M\times M\times[0,\infty)$, and $Z(x,x,t)=-\varepsilon(1+t)<0$ for each $x\in M$ and $t\in[0,T)$.  It follows that if $Z$ ever becomes positive, there exists a first time $t_0>0$ and points $x_0\neq y_0$ in $M$ such that $Z(x_0,y_0,t_0)=0$.   There are two possibilities:  Either both $x_0$ and $y_0$ are in the interior of $M$, or at least one of them (say $x_0$) lies in the boundary $\partial M$.

We deal with the first case first:
Clearly $Z(x,y,t)\leq 0$ for all $x,y\in M$ and $t\in[0,t_0]$.   In particular if we let $v(x,y)=2\varphi\left(\frac{d(x,y)}{2},t_0\right)$ and $\psi(x,y)=u(y,t)-u(x,t)-\varepsilon(1+t_0)$ then 
$$
\psi(x,y)\leq v(x,y)
$$
for all $x,y\in M$, while $\psi(x_0,y_0)=v(x_0,y_0)$.  Since $\psi$ is smooth, by Theorem \ref{thm:dist-comp} we have 
$$
{\mathcal L}[\nabla^2\psi,\nabla\psi]\leq 2\left[\alpha(\varphi')\varphi''-(n-1)\tk\beta(\varphi')\varphi'\right]\big|_{\frac{d(x_0,y_0)}2}.
$$
Now we observe that since the mixed partial derivatives of $\nabla^2\psi$ all vanish, we have for any admissible $A$ in the definition of $\mathcal L$ that
$$
\text{\rm tr}\left(A\nabla^2\psi\right) = \left(a(Du)^{ij}\nabla_i\nabla_j u\right)\big|_{(y_0,t_0)}-\left(a(Du)^{ij}\nabla_i\nabla_ju\right)\big|_{(x_0,t_0)},
$$
and therefore
$$
{\mathcal L}[\nabla^2\psi,\nabla\psi] = \left(a(Du)^{ij}\nabla_i\nabla_j u\right)\big|_{(y_0,t_0)}-\left(a(Du)^{ij}\nabla_i\nabla_ju\right)\big|_{(x_0,t_0)}.
$$
It follows that
\begin{equation}\label{secvar}
a(Du)^{ij}\nabla_i\!\nabla_j u\big|_{(y_0,t_0\!)}\!\!\!-\!a(Du)^{ij}\nabla_i\!\nabla_ju\big|_{(x_0,t_0\!)}\!\leq 2\!\left[\!\alpha(\!\varphi')\varphi''\!\!-\!(\!n\!-\!1\!)\!\tk\beta(\varphi')\varphi'\right]\!\!\big|_{{d(x_0,y_0)}/2}.
\end{equation}
We also know that the time derivative of $Z$ is non-negative at $(x_0,y_0,t_0)$, since $Z(x_0,y_0,t)\leq 0$ for $t<t_0$:
\begin{equation}\label{tvar}
\frac{\partial Z}{\partial t}\big|_{(x_0,y_0,t_0)}=a(Du)^{ij}\nabla_i\nabla_j u\big|_{(y_0,t_0)}-\left.a(Du)^{ij}\nabla_i\nabla_j u\right|_{(x_0,t_0)}-2\frac{\partial\varphi}{\partial t}-\varepsilon\geq 0.
\end{equation}
Combining the inequalities \eqref{secvar} and \eqref{tvar} we obtain
$$
\frac{\partial\varphi}{\partial t}<\alpha(\varphi')\varphi''-(n-1)\tk\beta(\varphi')\varphi'
$$
where all terms are evaluated at the point $d(x_0,y_0)/2$.  This contradicts the assumption \ref{1deqn} in Theorem \ref{thm:moc}.  

Now we consider the second case, where $x_0\in\partial M$.  
Under this assumption that $\partial M$ is convex there exists \cite{BGS} a length-minimizing geodesic $\gamma:\ [0,d]\to M$ from $x_0$ to $y_0$, such that $\gamma(s)$ is in the interior of $M$ for $0<s<d$ and $\gamma'(0)\cdot \nu(x_0)>0$, where $\nu(x_0)$ is the inward-pointing unit normal to $\partial M$ at $x_0$.  We compute
$$
\frac{d}{ds}Z(\exp_{x_0}(s\nu(x_0)),y_0,t_0) = -\nabla_{\nu(x_0)}u-\varphi'(d/2)\nabla d(\nu(x_0),0)
=\varphi'(d/2)\gamma'(0)\cdot \nu(x_0)\geq 0.
$$
In particular $Z(\exp_{x_0}(s\nu(x_0)),y_0,t_0)>0$ for all small positive $s$, contradicting the fact that $Z(x,y,t_0)\leq 0$ for all $x,y\in M$.

Therefore $Z$ remains negative for all $(x,y)\in M$ and $t\in[0,T)$.  Letting $\varepsilon$ approach zero proves the theorem.
\end{proof}

\section{The eigenvalue lower bound}

Now we provide the proof of the sharp lower bound on the first eigenvalue (Theorem \ref{first eigenvalue estimate}), which follows very easily from the modulus of continuity estimate from Theorem \ref{thm:moc}.

\begin{proposition}  \label{mu as bound} For $M$ and $u$ as in Theorem 1 applied to the heat equation ($\alpha\equiv\beta\equiv1$ in \eqref{eq:formofa}), we have the oscillation estimate
$$|u(y,t)-u(x,t)|\le C\E^{-\mu t},$$
where $C$ depends on the modulus of continuity of $u(\cdot,0)$, and  $\mu$ is the smallest positive eigenvalue of the Sturm-Liouville equation 
\begin{gather} \label{SL equation}  \begin{split}
\Phi''-(n-1)\tk \Phi' +\mu\Phi=\frac1{\ck^{n-1}}\left(\Phi' \ck^{n-1}\right)' +\mu \Phi&=0 \text{ on }[-D/2,D/2],\\
\Phi'(\pm D/2 )&=0.   
\end{split}
\end{gather}
\end{proposition}

\begin{proof}
The eigenfunction-eigenvalue pair $(\Phi,\mu)$ is defined as follows:  For any $\sigma\in\RR$ we define
$\Phi_\sigma(x)$ to be the solution of the initial value problem
\begin{align*}\Phi_\sigma''-(n-1)\tk\Phi_\sigma'+\sigma\Phi_\sigma&=0;\\
\Phi_\sigma(0)&=0;\\
\Phi'_\sigma(0)&=1.
\end{align*}
Then $\mu=\sup\{\sigma:\ x\in[-D/2,D/2]\Longrightarrow \Phi'_\sigma(x)>0\}$.  In particular, for $\sigma<\mu$ the function $\Phi_\sigma$ is strictly increasing on $[-D/2,D/2]$, and $\Phi_\sigma(x)$ is decreasing in $\sigma$ and converges smoothly to $\Phi(x)=\Phi_\mu(x)$ as $\sigma$ approaches $\mu$ for $x\in(0,D/2]$ and $0<\sigma<\mu$.  

Now we apply Theorem \ref{thm:moc}:  Since $\Phi$ is smooth, has positive derivative at $x=0$ and is positive for $x\in(0,D/2]$, there exists $C>0$ such that $C\Phi$ is a modulus of continuity for $u(.,0)$.  Then for each $\sigma\in(0,\mu)$, $\varphi_0=C\Phi_\sigma$ is also a modulus of continuity for $u(.,0)$, with $\varphi_0(0)=0$ and $\varphi'_0>0$.   Defining $\varphi(x,t) = C\Phi_\sigma(x)\E^{-\sigma t}$, all the conditions of Theorem  \ref{thm:moc} are satisfied, and we deduce that $\varphi(.,t)$ is a modulus of continuity for $u(.,t)$ for each $t\geq 0$.  Letting $\sigma$ approach $\mu$, we deduce that $C\Phi\E^{-\mu t}$ is also a modulus of continuity.  That is, for all $x,y$ and $t\geq 0$
$$
\left|u(y,t)-u(x,t)\right|\leq C\E^{-\mu t}\Phi\left(\frac{d(x,y)}{2}\right)\leq C\sup\Phi\ \E^{-\mu t}.
$$
\end{proof}

\paragraph{\textbf{Proof of Theorem \ref{first eigenvalue estimate}.}}    Observe that if $(\varphi,\lambda)$ is the first eigenfunction-eigenvalue pair, then $u(x,t)=e^{-\lambda t}\varphi(x)$ satisfies the heat equation on $M$ for all $t>0$.   From Proposition \ref{mu as bound}, we have 
$|u(y,t)-u(x,t)|\le Ce^{-\mu t}$ and so $|\varphi(y)-\varphi(x)|\le Ce^{-(\mu-\lambda)t}$ for all $x,y\in M$ and $t>0$.       Since $\varphi$ is non-constant, letting $t\rightarrow\infty$ implies that $\mu-\lambda\le 0$.    \qed

\smallskip
\section{Sharpness of the estimates}\label{sec:examples}

In the previous section we proved that $\lambda_1(D,\kappa,n)\geq \mu$.  To complete the proof of Theorem \ref{first eigenvalue estimate} we must prove that $\lambda_1(D,\kappa,n)\leq \mu$.  To do this we construct examples of Riemannian manifolds with given diameter bounds and Ricci curvature lower bounds such that the first eigenvalue is as close as desired to $\mu$.  The construction is similar to that given in \cite{Kroeger} and \cite{BakryQian}, but we include it here because the construction also produces examples proving that the modulus of continuity estimates of Theorem \ref{thm:moc} are sharp.

Fix $\kappa$ and $D$, and let $M=S^{n-1}\times [-D/2,D/2]$ with the metric
$$
g = ds^2 + a\ck^2(s)\bar g
$$
where $\bar g$ is the standard metric on $S^{n-1}$, and $a>0$.  The Ricci curvatures of this metric are given by
\begin{align*}
\text{\rm Ric}(\partial_s,\partial_s)&=(n-1)\kappa;\\
\text{\rm Ric}(\partial_s,v)&=0\quad\text{for\ }v\in TS^{n-1};\\
\text{\rm Ric}(v,v)&=\left((n-1)\kappa + (n-2)\frac{\frac{1}{a}-\kappa}{\ck^2}\right)|v|^2\quad\text{for\ }v\in TS^{n-1}.
\end{align*}
In particular the lower Ricci curvature bound $\text{\rm Ric}\geq (n-1)\kappa$ is satisfied for any $a$ if $\kappa\leq 0$ and for $a\leq 1/\kappa$ if $\kappa>0$.

To demonstrate the sharpness of the modulus of continuity estimate in Theorem \ref{thm:moc}, we construct solutions of equation \eqref{eq:flow} on $M$ which satisfy the conditions of the Theorem, and satisfy the conclusion with equality for positive times:  Let $\varphi_0:\ [0,D/2]$ be as given in the Theorem, and extend by odd reflection to $[-D/2,D/2]$ and define $\varphi$ to be the solution of the initial-boundary value problem
\begin{align*}
\frac{\partial\varphi}{\partial t} &= \alpha(\varphi')\varphi''+(n-1)\tk\beta(\varphi')\varphi';\\
\varphi(x,0)&=\varphi_0(x);\\
\varphi'(\pm D/2,t)&= 0.
\end{align*}
Now define $u(z,s,t) = \varphi(s,t)$ for $s\in[-D/2,D/2]$, $z\in S^{n-1}$, and $t\geq 0$.  Then a direct calculation shows that $u$ is a solution of equation \eqref{eq:flow} on $M$.  
If $\varphi_0$ is concave on $[0,D/2]$, then we have $|\varphi_0(a)-\varphi_0(b)|\leq 2\varphi_0\left(\frac{|b-1|}{2}\right)$ for all $a$ and $b$ in $[-D/2,D/2]$.  For our choice of $\varphi$ this also remains true for positive times.   Note also that for any $w,z\in S^{n-1}$ and $a,b\in[-D/2,D/2]$ we have
$d((w,a),(z,b))\geq |b-a|$.  Therefore we have
$$
|u(w,a,t)-u(z,b,t)|=|\varphi(a,t)-\varphi(b,t)|\leq 2\varphi\left(\frac{|b-a|}{2},t\right)\leq 
2\varphi\left(\frac{d((w,a),(z,b))}{2},t\right),
$$
so that $\varphi(.,t)$ is a modulus of continuity for $u(.,t)$ as claimed.  Furthermore, this holds with equality whenever $w=z$ and $b=-a$, so there is no smaller modulus of continuity and the estimate is sharp.

Now we proceed to the sharpness of the eigenvalue estimate.   On the manifold constructed above we have an explicit eigenfunction of the Laplacian, given by $\varphi(z,s) = \Phi(s)$ where $\Phi$ is the first eigenfunction of the one-dimensional Sturm-Liouville problem given in Proposition \ref{mu as bound}.
That is, we have $\lambda_1(M,g)\leq \mu$.  In this example we have the required Ricci curvature lower bound, and the diameter approaches $D$ as $a\to 0$.  Since $\mu$ depends continuously on $D$, the result follows.

A slightly more involved construction shows that examples of compact manifolds without boundary can also be constructed showing that the eigenvalue bound is sharp even in the smaller class of manifolds without boundary.  This is achieved by smoothing attaching a small spherical region at the ends of the above examples (see the similar construction in \cite{Andrews-Ni}*{Section 2}).

\section{Implications for the `Li conjecture'}\label{sec:Li}

In this section we mention some implications of the sharp eigenvalue estimate and a conjecture attributed to Peter Li:  The result of Lichnerowicz \cite{Lich} is that $\lambda_1\geq n\kappa$ whenever $\text{\rm Ric}\geq (n-1)\kappa g_{ij}$ (so that, by the Bonnet-Myers estimate, $D\leq \frac{\pi}{\sqrt{\kappa}}$).  The Zhong-Yang estimate \cite{ZY} gives $\lambda_1\geq \frac{\pi^2}{D^2}$ for $\text{\rm Ric}\geq 0$.  Both of these are sharp, and the latter estimate should also be sharp as $D\to 0$ for any lower Ricci curvature bound.
Interpolating linearly (in $\kappa$) between these estimates we obtain Li's conjecture
$$
\lambda_1\geq \frac{\pi^2}{D^2}+(n-1)\kappa.
$$
By construction this holds precisely at the endpoints $\kappa\to 0$ and $\kappa\to\frac{\pi^2}{D^2}$.

Several previous attempts to prove such inequalities have been made, particularly towards proving inequalities of the form $\lambda_1\geq \frac{\pi^2}{D^2}+a\kappa$ for some constant $a$, which are linear in $\kappa$ and have the correct limit as $\kappa\to 0$.  These include works of  DaGang Yang \cite{yang-eigenvalue}, Jun Ling \cite{ling-eigenvalues-2007} and Ling and Lu \cite{ ling-lu-eigenvalues}, the latter showing that $\alpha=\frac{34}{100}$ holds.    These are all superseded by the result of Shi and Zhang \cite{SZ} which proves $\lambda_1\geq \sup_{s\in(0,1)}\left\{4s(1-s)\frac{\pi^2}{D^2}+(n-1)s\kappa\right\}$, so in particular $\lambda_1\geq \frac{\pi^2}{D^2}+\frac{n-1}{2}\kappa$ by taking $s=\frac12$.

We remark here that the inequality with $a=\frac{n-1}{2}$ is the best possible of this kind, and in particular the Li conjecture is false.  This can be seen by computing an asymptotic expansion for the sharp lower bound $\mu$ given by Theorem \ref{first eigenvalue estimate}:  For fixed $D=\pi$ we perturb about $\kappa=0$ (as in \cite{Andrews-Ni}*{Section 4}), obtaining 
$$
\mu = 1+\frac{(n-1)}{2}\kappa+O(\kappa^2).
$$
By scaling this amounts to the estimate
$$
\mu = \frac{\pi^2}{D^2}+\frac{(n-1)}{2}\kappa+O(\kappa D^2).
$$
Since the lower bound $\lambda_1\geq \mu$ is sharp, this shows that the inequality $\lambda_1\geq \frac{\pi^2}{D^2}+a\kappa$ is false for any $a>\frac{(n-1)}{2}$, and in particular for $a=n-1$.


\begin{bibdiv}
\begin{biblist}

\bib{AC1}{article}{
   author={Andrews, Ben},
   author={Clutterbuck, Julie},
   title={Lipschitz bounds for solutions of quasilinear parabolic equations
   in one space variable},
   journal={J. Differential Equations},
   volume={246},
   date={2009},
   number={11},
   pages={4268--4283}
   }

\bib{AC2}{article}{
   author={Andrews, Ben},
   author={Clutterbuck, Julie},
   title={Time-interior gradient estimates for quasilinear parabolic
   equations},
   journal={Indiana Univ. Math. J.},
   volume={58},
   date={2009},
   number={1},
   pages={351--380},
  }

\bib{AC3}{article}{
   author={Andrews, Ben},
   author={Clutterbuck, Julie},
   title={Proof of the fundamental gap conjecture},
   journal={J. Amer. Math. Soc.},
   volume={24},
   date={2011},
   number={3},
   pages={899--916},
   }
   
\bib{Andrews-Ni}{article}{
	author={Andrews, Ben},
	author={Ni, Lei},
	title={Eigenvalue comparison on Bakry-Emery manifolds},
	eprint={http://arxiv.org/abs/1111.4967}
	}
	
\bib{BakryQian}{article}{
   author={Bakry, Dominique},
   author={Qian, Zhongmin},
   title={Some new results on eigenvectors via dimension, diameter, and
   Ricci curvature},
   journal={Adv. Math.},
   volume={155},
   date={2000},
   number={1},
   pages={98--153},
  }
  
\bib{BGS}{article}{
   author={Bartolo, Rossella},
   author={Germinario, Anna},
   author={S{\'a}nchez, Miguel},
   title={Convexity of domains of Riemannian manifolds},
   journal={Ann. Global Anal. Geom.},
   volume={21},
   date={2002},
   number={1},
   pages={63--83},
  }

\bib{Kroeger}{article}{
   author={Kr{\"o}ger, Pawel},
   title={On the ranges of eigenfunctions on compact manifolds},
   journal={Bull. London Math. Soc.},
   volume={30},
   date={1998},
   number={6},
   pages={651--655},
  }
  
\bib{Li-ev}{article}{label={Li},
   author={Li, Peter},
   title={A lower bound for the first eigenvalue of the Laplacian on a
   compact manifold},
   journal={Indiana Univ. Math. J.},
   volume={28},
   date={1979},
   number={6},
   pages={1013--1019},
  }
  
\bib{LiYau}{article}{
   author={Li, Peter},
   author={Yau, Shing Tung},
   title={Estimates of eigenvalues of a compact Riemannian manifold},
   conference={
      title={Geometry of the Laplace operator (Proc. Sympos. Pure Math.,
      Univ. Hawaii, Honolulu, Hawaii, 1979)},
   },
   book={
      series={Proc. Sympos. Pure Math., XXXVI},
      publisher={Amer. Math. Soc.},
      place={Providence, R.I.},
   },
   date={1980},
   pages={205--239},
}

\bib{Lich}{book}{
   author={Lichnerowicz, Andr{\'e}},
   title={G\'eom\'etrie des groupes de transformations},
   publisher={Travaux et Recherches Math\'ematiques, III. Dunod, Paris},
   date={1958},
   }

\bib{ling-eigenvalues-2007}{article}{author={Ling, Jun},
   title={The first eigenvalue of a closed manifold with positive Ricci
   curvature},
   journal={Proc. Amer. Math. Soc.},
   volume={134},
   date={2006},
   number={10},
   pages={3071--3079},}

\bib{ling-lu-eigenvalues}{article}{author={Ling, Jun},
   author={Lu, Zhiqin},
   title={Bounds of eigenvalues on Riemannian manifolds},
   conference={
      title={Trends in partial differential equations},
   },
   book={
      series={Adv. Lect. Math. (ALM)},
      volume={10},
      publisher={Int. Press, Somerville, MA},
   },
   date={2010},
   pages={241--264},
 }

\bib{Ni}{article}{
	author={Ni, Lei},
	title={Estimates on the modulus of expansion for vector fields solving nonlinear equations},
	eprint={http://arxiv.org/abs/1107.2351}
	}
	
\bib{SZ}{article}{
   author={Shi, Yu Min},
   author={Zhang, Hui Chun},
   title={Lower bounds for the first eigenvalue on compact manifolds},
   language={Chinese, with English and Chinese summaries},
   journal={Chinese Ann. Math. Ser. A},
   volume={28},
   date={2007},
   number={6},
   pages={863--866},
  }

\bib{yang-eigenvalue}{article}{
   author={Yang, DaGang},
   title={Lower bound estimates of the first eigenvalue for compact
   manifolds with positive Ricci curvature},
   journal={Pacific J. Math.},
   volume={190},
   date={1999},
   number={2},
   pages={383--398},
   }

\bib{ZY}{article}{
   author={Zhong, Jia Qing},
   author={Yang, Hong Cang},
   title={On the estimate of the first eigenvalue of a compact Riemannian
   manifold},
   journal={Sci. Sinica Ser. A},
   volume={27},
   date={1984},
   number={12},
   pages={1265--1273},
  }

\end{biblist}
\end{bibdiv}
\end{document}